\documentclass{amsart}

\usepackage[a4paper,hmargin=3.5cm,vmargin=4cm]{geometry}
\usepackage{amsfonts,amssymb,amscd,amstext}
\usepackage{graphicx}
\usepackage[dvips]{epsfig}
\usepackage{mathpazo}
\usepackage{enumerate}

\newcommand{\dist}{\operatorname{dist}}

\def\div{\mathfrak{Div}}

\def\Hgot{\mathfrak{H}}

\def\r{\mathbb{R}}
\def\n{\mathbb{N}}
\def\c{\mathbb{C}}

\def\d{\mathbb{D}}

\def\z{\mathbb{Z}}
\def\p{\mathbb{P}}

\def\cp{\mathbb{CP}}

\def\Ncal{\mathcal{N}}

\def\Fcal{\mathcal{F}}
\def\Kcal{\mathcal{K}}
\def\Bcal{\mathcal{B}}
\def\Hcal{\mathcal{H}}
\def\Acal{\mathcal{A}}
\def\Pcal{\mathcal{P}}
\def\Ucal{\mathcal{U}}

\def\Qcal{\mathcal{Q}}
\def\Ical{\mathcal{I}}

\def\ftt{{\tt f}}
\def\ptt{{\tt p}}
\def\nsc{{\text{\sc n}}}

\newtheorem{theorem}{Theorem}[section]

\newtheorem{claim}[theorem]{Claim}
\newtheorem{lemma}[theorem]{Lemma}
\newtheorem{corollary}[theorem]{Corollary}

\theoremstyle{definition}
\newtheorem{remark}[theorem]{Remark}
\newtheorem{definition}[theorem]{Definition}

\numberwithin{equation}{section}
\numberwithin{figure}{section}

\begin{document}

\title[Harmonic mappings  and minimal immersions in $\r^\nsc$]
{Harmonic mappings and conformal minimal immersions of Riemann surfaces into $\r^\nsc$}

\author[A.~Alarc\'{o}n]{Antonio Alarc\'{o}n$^\dagger$}
\address{Departamento de Matem\'{a}tica Aplicada \\ Universidad de Murcia
\\ E-30100 Espinardo, Murcia \\ Spain}
\email{ant.alarcon@um.es}

\author[I.~Fern\'{a}ndez]{Isabel Fern\'{a}ndez$^\ddagger$}
\address{Departamento de Matem\'{a}tica Aplicada I \\ Universidad de Sevilla
\\ E-41012 Sevilla \\ Spain}
\email{isafer@us.es}

\author[F.J.~L\'{o}pez]{Francisco J. L\'{o}pez$^\dagger$}
\address{Departamento de Geometr\'{\i}a y Topolog\'{\i}a \\
Universidad de Granada \\ E-18071 Granada \\ Spain}
\email{fjlopez@ugr.es}


\thanks{$^\dagger$ Research partially
supported by MCYT-FEDER research project MTM2007-61775 and Junta
de Andaluc\'{i}a Grant P09-FQM-5088}
\thanks{$^\ddagger$ Research partially
supported by MCYT-FEDER research project MTM2007-64504 and Junta
de Andaluc\'{i}a Grant P09-FQM-5088}


\begin{abstract}
We prove that for any open Riemann surface $\mathcal{N},$ natural number $\nsc \geq 3,$ non-constant harmonic map $h:\mathcal{N} \to \r^{\nsc-2}$ and holomorphic 2-form $\Hgot$ on $\Ncal,$ there exists a weakly complete harmonic map $X=(X_j)_{j=1,\ldots,\nsc}:\mathcal{N} \to \r^\nsc$ with Hopf differential $\Hgot$ and  $(X_j)_{j=3,\ldots,\nsc}=h.$ In particular, there exists a complete conformal minimal immersion $Y=(Y_j)_{j=1,\ldots,\nsc}:\mathcal{N} \to \r^\nsc$ such that $(Y_j)_{j=3,\ldots,\nsc}=h.$

As some consequences of these results:
\begin{itemize}
\item There exists complete full non-decomposable minimal surfaces with arbitrary conformal structure and whose generalized Gauss map is non-degenerate  and fails to intersect $\nsc$ hyperplanes of $\cp^{\nsc-1}$ in general position.
\item There exists complete non-proper embedded minimal surfaces in $\r^\nsc,$ $\forall\,\nsc>3.$
\end{itemize}
\end{abstract}

\maketitle

\thispagestyle{empty}


\section{Introduction}

In this paper we  use methods coming from the study of minimal surfaces to construct harmonic mappings of Riemann surfaces into $\r^\nsc$ with prescribed geometry.   A basic reference for this topic is, for instance,  Klotz's work \cite{tilla}.  

Our main result states  (see Corollary \ref{co:harmonic}):

\begin{quote} {\bf Theorem A.} {\em For any open Riemann surface $\mathcal{N},$ natural number $\nsc \geq 3,$ non-constant harmonic map $h:\mathcal{N} \to \r^{\nsc-2}$ and holomorphic 2-form $\Hgot$ on $\Ncal,$ there exists a weakly complete harmonic map $X=(X_j)_{j=1,\ldots,\nsc}:\mathcal{N} \to \r^\nsc$ with Hopf differential $\Hgot$ and   $(X_j)_{j=3,\ldots,\nsc}=h.$}
\end{quote}

Recall that the Hopf differential $Q_X$ of a harmonic map $X:\Ncal\to\r^\nsc$ is the holomorphic 2-form given by $Q_X:=\langle\partial_z X,\partial_z X\rangle,$ where $\partial_z$ means complex differential. By definition,  $X$ is said to be weakly complete if  $\Gamma_X:=|\partial_z X|^2$ is a complete conformal Riemannian metric in $\Ncal$ (see \cite{tilla}).

The fact that conformal minimal immersions are harmonic maps strongly influences the global theory of this kind of surfaces. It is well known that a harmonic immersion $X:\Ncal \to \r^\nsc$ is minimal if and only if it is conformal, or equivalently, $Q_X=0.$ Weakly completeness is equivalent to Riemannian completeness under minimality assumptions. The geometry of complete minimal surfaces in $\r^\nsc,$ specially those properties regarding the Gauss map, has been the object of extensive study over the last past decades (see for instance \cite{O1,CO,C,F3,R}). 

In the recent paper \cite{AFL}, the authors constructed complete  minimal surfaces in $\r^3$ with arbitrarily prescribed conformal structure and non-constant third coordinate function (see also \cite{AF}). As a consequence, any open Riemann surface admits a complete conformal minimal immersion in $\r^3$ whose Gauss map omits two antipodal points of the unit sphere. 
Theorem A let us generalize these results to  arbitrary higher dimensions (see Corollary \ref{co:minimal}):
\begin{quote}
{\bf Theorem B.} {\em  For any open Riemann surface $\mathcal{N},$ natural number $\nsc \geq 3$ and  non-constant harmonic map $h:\mathcal{N} \to \r^{\nsc-2},$ 
there exists a complete conformal minimal immersion $X=(X_j)_{j=1,\ldots,\nsc}:\mathcal{N} \to \r^\nsc$ with $(X_j)_{j=3,\ldots,\nsc}=h.$}
\end{quote}

Under some compatibility conditions depending on the map $h,$ the flux map of the immersion $X$ can be also prescribed. Recall that the flux map of a conformal minimal immersion $X:\mathcal{N} \to \r^\nsc$ is given by $\ptt_X(\gamma)=\mbox{Im}\int_{\gamma} \partial_z X$ for all $\gamma\in\Hcal_1(\Ncal,\z).$ In particular, if $h$ is the real part of a holomorphic map $H:\Ncal \to \c^{\nsc-2},$ Theorem B provides a complete null holomorphic curve $F=(F_j)_{j=1,\ldots,\nsc}:\Ncal \to \c^\nsc$  such that $(F_j)_{j=3,\ldots,\nsc}=H.$ Likewise, $Y=(F_j)_{j=2,\ldots,\nsc}:\Ncal \to \c^{\nsc-1}$ is a complete holomorphic immersion whose last $\nsc-2$ coordinates coincide with $H.$

Theorem B also includes some information about the Gauss map of minimal surfaces in $\r^{\nsc}.$ Given a conformal minimal immersion $X:\Ncal \to \r^\nsc,$ its generalized Gauss map $G_X:\Ncal \to \c \p^{\nsc-1},$ $G_X(P)=\partial_z X(P),$ is holomorphic and takes values on the complex hyperquadric $\{\sum_{j=1}^{\nsc} w_j^2=0\}.$ Chern and Osserman \cite{C,CO} showed that if $X$ is complete then either $X(\Ncal)$ is a plane or  $G_X(\Ncal)$ intersects a dense set of complex hyperplanes. Even more, Ru \cite{R} proved that if $X$ is complete and non-flat then $G_X$ cannot omit more than $\nsc(\nsc + 1)/2$ hyperplanes in $\cp^{\nsc-1}$ located in general position (see also the works of Fujimoto  \cite{F2,F3} for a good setting). Under the non-degeneracy assumption on $G_X,$ this upper bound is sharp for some values of $\nsc,$ see \cite{F4}. However, the number of exceptional hyperplanes  strongly depends on the underlying conformal structure of the surface. Indeed, Ahlfors  \cite{A} proved that any holomorphic map $G:\c\to\cp^{\nsc-1},$ $\nsc \geq 3,$ avoiding $\nsc+1$ hyperplanes of $\cp^{\nsc-1}$ in general position is degenerate, that is to say, $G(\c)$ lies in a proper projective subspace of $\cp^{\nsc-1}$ (see  \cite[Chapter 5, $\S$5]{W} and \cite{F} for further generalizations). So, it is natural to wonder whether any open Riemann surface admits a complete conformal minimal immersion in $\r^\nsc$ whose generalized Gauss map is non-degenerate and  omits $\nsc$ hyperplanes in general position. An affirmative answer to this question can be found in the following (see Corollary \ref{co:omite})

\begin{quote}
{\bf Theorem C.} {\em Let $\mathcal{N}$ be an open Riemann surface, and let $\ptt:\Hcal_1(\Ncal,\z)\to\r^\nsc$ be a group morphism, $\nsc\geq 3.$

Then there exists a complete conformal full non-decomposable minimal immersion $X:\Ncal\to\r^\nsc$ with $\ptt_X=\ptt$ and whose generalized Gauss map is non-degenerate and omits $\nsc$ hyperplanes in general position.}
\end{quote}

On the other hand, Theorem B leads to some interesting consequences regarding Calabi-Yau conjectures. The embedded Calabi-Yau  problem for minimal surfaces asks for the existence of complete bounded  embedded minimal surfaces in $\r^3.$  Complete embedded minimal surfaces in $\r^3$ with finite genus and countably many ends are proper in space \cite{MPR,CM}. However, this result fails to be true for arbitrary higher dimensions. For instance,  taking $\Ncal$ the unit disc $\d$ in $\c$ and  $h:\d \to \r^2$ the map $h(z)=(\rm{Re}(z),\rm{Im}(z)),$  Theorem B generates complete non-proper embedded minimal discs in $\r^4$  (so in $\r^\nsc$ for all $\nsc \geq 4$), see Corollary \ref{co:CM} for more details.

The paper is laid out as follows. In Section \ref{sec:preli} we introduce the necessary background and notations. In Section \ref{sec:approx} we prove a basic approximation result by holomorphic 1-forms in open Riemann surfaces (Lemma \ref{lem:ap}), which is the key tool for proving our main results.  Finally, in Section \ref{sec:main} we state and prove Theorems A, B and C. It is worth mentioning that all these theorems  actually follows from the more general result Theorem \ref{th:main} in Section \ref{sec:main}.


\section*{Acknowledgments}
The authors would like to thank Jos\'{e} A. G\'{a}lvez for directing their attention to the paper \cite{tilla}.

\section{Preliminaries} \label{sec:preli}

Given a topological surface $M,$  $\partial M$ will denote the one dimensional topological manifold determined by the boundary points of $M.$ Given $S \subset M,$ call by $S^\circ$ and $\overline{S}$  the interior and the closure  of $S$  in $M,$ respectively. Open connected subsets of $M-\partial M$ will be called {\em domains}, and those proper topological subspaces of $M$ being surfaces with boundary are said to be  {\em regions}. The surface $M$ is said to be {\em open} if it is non-compact and $\partial M =\emptyset.$

If $M$ is a Riemann surface, $\partial_z$ will denote the global complex operator given by $\partial_z|_U=\frac{\partial}{\partial w} dw$ for any conformal chart $(U,w)$ on $M.$

\begin{remark}
Throughout this paper $\Ncal$ and $\nsc$ will denote a fixed but arbitrary open Riemann surface and natural number greater than or equal to three, respectively.
\end{remark}

Let $S$ denote a subset of $\mathcal{ N},$ $S \neq \mathcal{ N}.$ We denote by  ${\mathcal{ F}_0}(S)$ as the space of continuous functions $f:S \to{\c}$ which are holomorphic on an open neighborhood  of $S$ in $\mathcal{ N}.$ Likewise, $\mathcal{ F}_0^*(S)$  will denote the space of continuous functions $f:S \to \c$ being holomorphic on $S^\circ.$ 

As usual, a 1-form $\theta$ on $S$ is said to be of type $(1,0)$ if for any conformal chart $(U,z)$ in $ \mathcal{ N},$ $\theta|_{U \cap S}=h(z) dz$ for some function $h:U \cap S \to \c.$  We denote by $\Omega_0(S)$ as the space of holomorphic 1-forms on an open neighborhood of $S$ in $\mathcal{ N}.$   We call $\Omega_0^*(S)$ as the space of 1-forms $\theta$ of type $(1,0)$ on $S$ such that $(\theta|_{U})/dz \in \mathcal{ F}_0^*(S \cap U)$ for any conformal chart $(U,z)$ on $\mathcal{ N}.$

We denote by $\mho_0(S)$ as the space of holomorphic 2-forms on an open neigborhood of $S$ in $\Ncal.$

Let $\div(S)$ denote the free commutative group of divisors of $S$ with multiplicative notation.
A divisor $D \in \div(S)$ is said to be {\em integral} if $D=\prod_{i=1}^n Q_i^{n_i}$ and $n_i\geq 0$ for all $i.$ Given $D_1,$ $D_2 \in \div(S),$ we write $D_1 \geq D_2$ if and only if $D_1 D_2^{-1}$ is  integral. For any $f \in \mathcal{ F}_0(S)$ we denote by $(f)$ its associated integral divisor of zeros in $S.$ Likewise we define $(\theta)$  for any  $\theta \in \Omega_0(S).$

In the sequel we will assume that $S$ is a {\em compact subset} of $\Ncal.$

A compact Jordan arc in $\mathcal{ N}$ is said to be analytical if it is contained in an open analytical Jordan arc in $\mathcal{ N}.$ 
By definition, a connected component $V$ of $\mathcal{ N}-S$ is said to be {\em bounded}  if $\overline{V}$ is compact, where $\overline{V}$ is the closure of $V$ in $\mathcal{ N}.$ Moreover, a subset $K\subset\Ncal$ is said to be {\em Runge} (in $\Ncal$) if $\Ncal-K$ has no bounded components.
\begin{definition}
A compact subset $S\subset \mathcal{ N}$ is said to be {\em admissible} if and only if (see Figure \ref{fig:admi}):
\begin{itemize}
\item $S$ is Runge, 
\item $M_S:=\overline{S^\circ}$  consists of a finite
collection of pairwise disjoint compact regions in $\mathcal{ N}$
with $\mathcal{ C}^0$ boundary, \item $C_S:=\overline{S-M_S}$
consists of a finite collection of pairwise disjoint analytical
Jordan arcs, and \item any component $\alpha$ of $C_S$  with an
endpoint $P\in M_S$ admits an analytical extension $\beta$ in
$\mathcal{ N}$ such that the unique component of $\beta-\alpha$
with endpoint $P$ lies in $M_S.$
\end{itemize}
\end{definition}

\begin{figure}[ht]
    \begin{center}
    \scalebox{0.35}{\includegraphics{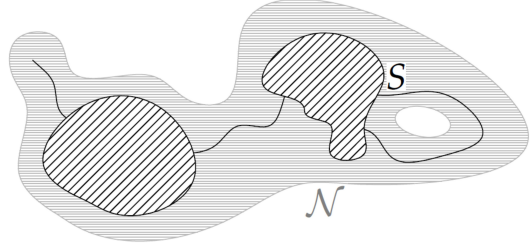}}
        \end{center}
\caption{An admissible subset.}\label{fig:admi}
\end{figure}

Let $W$ be a domain  in $\Ncal$, and let $S$ be either a compact region or an admissible subset in $\Ncal.$ $W$ is said to be a {\em tubular neighborhood} of $S$ if $S \subset W$ and  $W-S$ consists of a finite collection of pairwise disjoint open annuli. In addition, if $\overline{W}$ is a compact region isotopic to $W$ then $\overline{W}$ is said to be a {\em compact tubular neighborhood} of $S.$  Here {\em isotopic} means that $\jmath_*:\Hcal_1(W,\z) \to \Hcal_1(\overline{W},\z)$ is an isomorphism, where $\jmath:W \to \overline{W}$ is the inclusion map.

Let $W\subset\Ncal$ be a domain with $S\subset W.$
We shall say that a function  $f \in \mathcal{ F}_0^*(S)$ can be {\em uniformly approximated} on $S$ by functions in ${\mathcal{ F}_0}(W)$ if there exists  $\{f_n\}_{n \in \n} \subset {\mathcal{ F}_0}(W)$ such that $\{|f_n-f|\}_{n \in \n} \to 0$ uniformly on $S.$     A 1-form  $\theta \in \Omega_0^*(S)$ can be uniformly approximated on $S$ by 1-forms  in $\Omega_0(W)$ if there exists   $\{\theta_n\}_{n \in \n}\subset\Omega_0(W)$ such that $\{\frac{\theta_n-\theta}{dz}\}_{n \in \n} \to 0$ uniformly on $S \cap U,$ for any conformal closed disc $(U,dz)$ on $W.$

Given an admissible compact set $S\subset W,$ a function $f:S\to\c^n,$ $n\in\n,$ is said to be smooth if $f|_{M_S}$  admits a smooth extension $f_0$ to an open domain $V$ in $W$ containing $M_S,$ and for any component $\alpha$ of $C_S$ and any open analytical Jordan arc $\beta$ in $W$ containing $\alpha,$  $f$ admits an smooth extension $f_\beta$ to $\beta$ satisfying that $f_\beta|_{V \cap \beta}=f_0|_{V \cap \beta}.$ Likewise, an 1-form $\theta$ of type $(1,0)$ on $S$ is said to be smooth if for any closed conformal disc $(U,z)$ on $\Ncal$ such that $S \cap U$ is admissible, the function $\frac{\theta}{dz}$ is smooth on $S\cap U.$ Given a smooth $f\in\Fcal_0^*(S),$ we set $df \in \Omega_0^*(S)$ as the smooth 1-form given by $df|_{M_S}=d (f|_{M_S})$ and $df|_{\alpha \cap U}=(f \circ \alpha)'(x)dz|_{\alpha \cap U},$ where $(U,z=x+i y)$ is a conformal chart on $W$ such that $\alpha \cap U=z^{-1}(\r \cap z(U)).$ Obviously, $df|_\alpha(t)= (f\circ\alpha)'(t) dt$ for any component $\alpha$ of $C_S,$ where $t$ is any smooth parameter along $\alpha.$ This definition makes sense also for smooth functions with poles in $S^\circ.$

A smooth 1-form $\theta \in \Omega_0^*(S)$ is said to be {\em exact} if $\theta=df$ for some smooth $f \in \mathcal{ F}_0^* (S),$ or equivalently if $\int_\gamma \theta=0$ for all $\gamma \in \mathcal{ H}_1(S,\z).$


\subsection{Harmonic maps and minimal surfaces in $\r^\nsc$}\label{sec:minimal}

Given a non-constant harmonic map $X=(X_j)_{j=1,\ldots,\nsc}:\Ncal\to \r^\nsc,$ the holomorphic quadratic differential 
\[
Q_X:=\langle\partial_z X,\partial_z X\rangle=\sum_{j=1}^\nsc (\partial_z X_j)^2
\]
is said to be the Hopf differential of $X.$ We also consider the conformal metric, possibly with isolated singularities,
\[
\Gamma_X:=\frac{1}{2}\sum_{j=1}^\nsc |\partial_z X_j|^2.
\]
It is clear that $2\Gamma_X$ is greater than or equal to the  Riemannian metric on $\Ncal$ (possibly with singularities) induced by $X.$ When $X$ is an immersion then $\Gamma_X$ is a Riemannian metric, and if in addition $X$ is complete then  $\Gamma_X$ is complete as well \cite{tilla}. However, the reciprocal does not hold in general.
\begin{definition}\label{def:tilla}
We say that a harmonic map $X:\Ncal\to\r^\nsc$ is {\em weakly complete} (or complete  in the sense of \cite{tilla}) if $\Gamma_X$ is a complete metric in $\Ncal.$
\end{definition}

We also associate to $X$ the group morphism
\[
\ptt_X:\Hcal_1(\Ncal,\z)\to\r^\nsc,\quad \ptt_X(\gamma)={\rm Im}\int_\gamma \partial_z X.
\]

\begin{remark}\label{rem:Q=0}
If $Q_X= 0$ and $\Gamma_X$ never vanishes, then $X$ is a conformal minimal immersion, $\Gamma_X$ is the metric induced on $\Ncal$ by $X,$ and $\ptt_X$ is the flux map of $X.$
\end{remark}

If in addition $X$ is a conformal minimal immersion and we write $\partial_z X_j=f_j d\zeta$ in terms of a local parameter $\zeta$ on $\Ncal,$ $j=1,\ldots,\nsc,$ then the {\em (generalized) Gauss map} of $X$ is given by
\[
G_X:\Ncal\to \cp^{\nsc-1},\quad G_X(\zeta)=[(f_j(\zeta))_{j=1,\ldots,\nsc}],
\]
where $[w]$ is the class of $w$ in $\cp^{\nsc-1},$ $\forall w\in \c^\nsc.$ It is well known that $G_X$ is a holomorphic map taking values in the complex quadric $\{[(w_j)_{j=1,\ldots,\nsc}]\in\cp^{\nsc-1}\;|\; \sum_{j=1}^\nsc w_j^2=0\}.$ 

A set of hyperplanes in $\cp^{\nsc-1}$ is said to be in general position if each subset of $k$ hyperplanes, with $k \leq \nsc-1,$ has an
$(\nsc-1-k)$-dimensional intersection.

\begin{definition}[\cite{O}]  
Let $X:\Ncal \to \r^{\nsc}$ be a conformal minimal immersion.
\begin{itemize}
\item $X$ is said to be {\em decomposable} if, with respect to suitable rectangular coordinates in $\r^{\nsc},$ one has $\sum_{k=1}^m (\partial_z X_k)^2= 0$ for some $m<\nsc.$ 
\item $X$ is said to be {\em full} if $X(\Ncal)$ is contained in no hyperplane of $\r^\nsc.$ 
\item The Gauss map $G_X$ is said to be {\em degenerate} if $G_X(\Ncal)$ lies in a hyperplane of $\cp^{\nsc-1}.$
\end{itemize} 
\end{definition}

When $\nsc=3,$ decomposable, non-full and degenerate are equivalent. However, if one passes to higher dimensions then no two of these conditions are equivalent  (see \cite{O}).


\section{The Approximation Lemma} \label{sec:approx}

The next two lemmas are the key tools in the proof of the main result of this section (Lemma \ref{lem:ap}). They represent a slight generalization of Lemmas 2.4 and 2.5 in \cite{AL}.

From now on, $\imath$ denotes the imaginary unit and the symbol $\not\equiv 0$  means {\em non-identically zero}.

\begin{lemma}\label{lem:funaprox}
Let $W\subset \Ncal$ be a domain with finite topology and $S\subset\Ncal$ an admissible subset with $S\subset W.$
Consider $f\in \Fcal_0^*(S)\cap \Fcal_0(M_S)$ with $f|_{M_S}\not\equiv 0.$

Then $f$ can be uniformly approximated on $S$ by functions $\{f_n\}_{n \in \n}$ in $\Fcal_0(W)$ satisfying that $(f_n)=(f|_{M_S})$ on $W.$  In particular, $f_n$ never vanishes on $W-M_S$ for all $n.$
\end{lemma}
\begin{proof}
Let $\{M_k\}_{k\in\n}$ be a sequence of compact tubular neighborhoods of $M_S$ in $W$ such that $M_k \subset M_{k-1}^\circ$ for any $k,$ $\cap_{k \in  \n} M_k=M_S$ and $f$ holomorphically extends  (with the same name) to $M_1$ and has no zeros on $M_1-M_S$ (take into account that $f|_{M_S}\not\equiv 0$). Choose $M_k$ so that, in addition, the compact set $S_k:=M_k \cup C_S\subset W$ is admissible and $\alpha-M_k^\circ$ is a (non-empty) Jordan arc for any component $\alpha$ of $C_S.$ In particular, $M_{S_k}=M_k$ and $C_{S_k}= C_S-M_k^\circ,$ $k \in \n.$

For any $k\in\n$ take $g_k\in \Fcal_0^*(S_k)\cap \Fcal_0(M_{S_k})$ satisfying
\begin{itemize}
\item $g_k|_{M_{S_k}}=f|_{M_{S_k}},$ 
\item $g_k$ never vanishes on $S_k-S_k^\circ$ (recall that $f$ has no zeros on $M_1-M_S$), and 
\item the sequence $\{g_k|_S\}_{k \in \n}$ uniformly converges to $f$ on $S.$
\end{itemize}

The construction of such functions is standard, we omit the details. Since $g_k$ satisfies the hypotheses of Lemma 2.4 in \cite{AL}, it can be uniformly approximated on $S_k$ by a sequence $\{g_{k,n}\}_{n \in \n}\subset\Fcal_0(W)$ with $(g_{k,n})=(g_k|_{M_{S_k}})=(f|_{M_S})$ on $W,$ for any $k.$ A standard diagonal argument concludes the proof.
\end{proof}

\begin{lemma}\label{lem:formaprox}
Let $W\subset \Ncal$ be a domain with finite topology and $S\subset\Ncal$ an admissible subset with $S\subset W.$
Consider $\theta \in \Omega_0^*(S)\cap \Omega_0(M_S)$ with $\theta|_{M_S}\not\equiv 0.$

Then  $\theta$ can be uniformly approximated on $S$ by 1-forms $\{\theta_n\}_{n \in \n}$ in $\Omega_0(W)$ satisfying that $(\theta_n)=(\theta|_{M_S})$ on $W.$  In particular, $\theta_n$ never vanishes on $W-M_S$ for all $n.$
\end{lemma}
\begin{proof}
Let $\vartheta$ be a never vanishing 1-form in $\Omega_0(W).$ Define $f:=\theta/\vartheta
\in \Fcal_0^*(S)\cap \Fcal_0(M_S).$ By Lemma \ref{lem:funaprox}, $f$ can be
uniformly approximated on $S$ by a sequence
$\{f_n\}_{n \in \n}$ in $\Fcal_0(W)$ satisfying that $(f_n)=(f|_{M_S})$
on $W$ for all $n.$ It suffices to set $\theta_n:=f_n \vartheta,$ $n
\in \n.$
\end{proof}


\begin{lemma}\label{lem:ap}
Let $W\subset \Ncal$ be a domain with finite topology and $S\subset\Ncal$ an admissible subset with $S\subset W.$ Let $\Theta\in\mho_0(W)$ and $\Phi=(\phi_1,\phi_2)$ be a smooth pair in $\Omega_0^*(S)^2\cap \Omega_0(M_S)^2$ satisfying
$
\phi_1^2+\phi_2^2=\Theta|_{S}
$
and either of the following conditions:
\begin{enumerate}[{\rm (A)}]
\item $\phi_1|_{M_S}$ and $\phi_2|_{M_S}$ are linearly independent in $\Omega_0(M_S)$ and $\Theta$ has no zeros on $C_S.$
\item $\Theta= 0$ and $\phi_1|_{M_S}\not\equiv 0.$
\end{enumerate}

Then $\Phi$ can be uniformly approximated on $S$ by a sequence $\{\Phi_n=(\phi_{1,n},\phi_{2,n})\}_{n\in\n}\subset \Omega_0(W)^2$ satisfying
\begin{enumerate}[{\rm (a)}]
\item $\phi_{1,n}^2+\phi_{2,n}^2=\Theta,$
\item $\Phi_n-\Phi$ is exact on $S,$ and
\item the zeros of $\Phi_n$ on $W$ are those of $\Phi$ on $M_S$ (in particular, $\Phi_n$ never vanishes on $W-M_S$).
\end{enumerate}
\end{lemma}

\begin{proof}
Assume (A) holds.

\begin{claim}\label{cl:ceros}
Without loss of generality it can be assumed that $\phi_1,$ $\phi_2$ and $d\xi$ never vanish on $C_S,$ where $\xi:=\phi_1/\phi_2.$
\end{claim}
\begin{proof} Assume for a moment that the conclusion of the lemma holds when $\phi_1,$ $\phi_2$ and $d\xi$ never vanish on $C_S.$

Take a sequence $\{M_k\}_{k\in\n}$ as in the proof of Lemma \ref{lem:funaprox} such that $\Phi$ holomorphically extends (with the same name) to $M_1,$ and $\phi_1,$ $\phi_2$ and $d\xi$ never vanish on $M_1-M_S,$ for all $n$  (take into account (A)). Recall that $S_k:=M_k \cup C_S\subset W^\circ$ is an admissible set and $C_{S_k}= C_S-M_k^\circ,$ $k \in \n.$

Since $\Theta$ never vanishes on $C_S,$ which
consists of a finite collection of pairwise disjoint analytical
Jordan arcs, then we can find $\theta\in\Omega_0(C_S)$ with $\theta^2=\Theta|_{C_S}.$ Consider $f_j:C_S\to\c,$ $f_j=\phi_j/\theta,$ $j=1,2,$ and notice that $f_1^2+f_2^2=1$ and $\xi|_{C_S}=f_1/f_2.$ Consider a sequence $\{(f_{1,k},f_{2,k})\}_{k \in \n}$ of pairs of smooth functions on $C_{S_k}$ satisfying:
\begin{enumerate}[ $i)$]
\item $f_{1,k},$ $f_{2,k}$ and $d f_{1,k}$ never vanish on $C_{S_k},$
\item $f_{1,k}^2+f_{2,k}^2=1,$ 
\item the function $g_{j,k}$ given by $g_{j,k}|_{M_{S_k}}=f_j,$ $g_{j,k}|_{C_{S_k}}=f_{j,k},$ lies in $\Fcal_0^*(S_k)$ and is smooth, $j=1,2,$
\item $\{f_{j,k}\}_{k \in \n}$ uniformly converges to $f_j$ on $C_S,$ $j=1,2,$ and
\item $\Psi_k|_S-\Phi$ is exact on $S,$ where $\Psi_k:=(g_{j,k}\theta)_{j=1,2}\in \Omega_0^*(S_k)^2\cap \Omega_0(M_{S_k})^2.$ 
\end{enumerate}
The construction of this data is standard, we omit the details. Write $\Psi_k=(\psi_{j,k})_{j=1,2}$ and $\xi_k=\psi_{1,k}/\psi_{2,k}.$ From $i),$ $ii)$ and the definition of $\theta$ follow that $\psi_{1,k}^2+\psi_{2,k}^2=\Theta$ and $d\xi_k$ never vanishes on $C_{S_k}.$ Moreover, $iv)$ gives that $\{\Psi_k|_S\}_{k \in \n}$ uniformly converges to $\Phi$ on $S.$

By hypothesis, Lemma \ref{lem:ap} holds for any $\Psi_k,$ then there exists a sequence $\{\Psi_{k,n}\}_{n\in\n}$ uniformly converging to $\Psi_k$ on $S_k$ and satisfying (a), (b) and (c) of Lemma \ref{lem:ap} for $\Phi=\Psi_k$ and $S=S_k.$ Using that $\{\Psi_k|_S\}_{k \in \n}$ converges to $\Phi,$ the zeros of $\Psi_k$ in $M_{S_k}$ are those of $\Phi$ in $M_S,$ $v),$ and a standard diagonal argument, we can obtain a sequence satisfying the conclusion of the lemma,  proving the claim.
\end{proof}
In the sequel we will assume that  $\phi_1,$ $\phi_2$ and $d\xi$ never vanish on $C_S.$

Label $\eta=\phi_1-\imath \phi_2\in\Omega_0^*(S)\cap  \Omega_0(M_S)$ and observe that $\Theta/\eta=\phi_1+\imath\phi_2\in\Omega_0^*(S)\cap \Omega_0(M_S).$ Notice that  $(\Theta/\eta)|_{M_S},$ $\eta|_{M_S}\not\equiv 0,$ 
\[
\phi_1=\frac12\left(\eta+\frac{\Theta}{\eta}\right)\quad \text{and}\quad \phi_2=\frac{\imath}{2}\left(\eta-\frac{\Theta}{\eta}\right).
\]

Let $\Bcal_S$ be a homology basis of $\Hcal_1(S,\z)$ and label $\nu$ as its cardinal number. Consider in $\Fcal_0^*(S)$ the maximum norm and the Fr\'{e}chet differentiable map
\[
\Pcal:\Fcal_0^*(S)\to \c^{2\nu},\quad
\Pcal(f)=\left( \int_{c} \left( e^f\eta+e^{-f}\frac{\Theta}{\eta}-2\phi_1 \;,\; e^f\eta-e^{-f}\frac{\Theta}{\eta}+2\imath\phi_2 \right) \right)_{c \in \Bcal_S}.
\]
Label $\Acal:\Fcal_0^*(S) \to \c^{2\nu}$ as the Fr\'{e}chet derivative of $\Pcal$ at $0.$

\begin{claim}\label{cla:sobre}
$\Acal|_{\Fcal_0(W)}$ is surjective.
\end{claim}
\begin{proof}
Reason by contradiction and assume that $\Acal(\Fcal_0(W))$ lies in a complex subspace
$\Ucal=\{\big((x_c,y_c)\big)_{c \in \Bcal_S} \in \c^{2\nu}\,|\; \sum_{c \in \Bcal_S} \big(A_c x_c+B_c y_c\big)=0\},$ where
$A_c,B_c \in \c,$ $\forall c\in \Bcal_S,$  and 
\begin{equation}\label{eq:curvas}
\sum_{c \in \Bcal_S} \big(|A_c|+|B_c|\big)\neq 0.
\end{equation}
Then, writing $\Gamma_1= \sum_{c \in \Bcal_S} A_c\, c$ and $\Gamma_2= \sum_{c \in \Bcal_S} B_c\, c,$ we have
\begin{equation}\label{eq:sobre1}
\int_{\Gamma_1} f\phi_2+\imath\int_{\Gamma_2} f\phi_1=0,\quad \forall f\in\Fcal_0(W).
\end{equation}
Denote by $\Sigma=\{f\in\Fcal_0(W)\;|\; (f)\geq (\phi_2|_{M_S})^2\}$ (recall that $\phi_2$ never vanishes on $C_S$). Then for any $f\in \Sigma$ the function $df/\phi_2\in \Fcal_0^*(S)\cap\Fcal_0(M_S),$ so it can be uniformly approximated on $S$ by functions in $\Fcal_0(W).$ This fact is trivial when $f$ is constant, otherwise use Lemma \ref{lem:funaprox}. Hence equation \eqref{eq:sobre1} applies and gives
\begin{equation}\label{eq:sobre2}
0=\int_{\Gamma_2} \xi df = \int_{\Gamma_2} f\, d\xi,\quad \forall f\in\Sigma,
\end{equation}
where we have used integration by parts (notice that $f \xi,$  $\xi df$ and $f d\xi$ are smooth).

Suppose $\Gamma_2\neq 0$ and take $[\tau]\in \Hcal^1_{\text{hol}}(W)$ (the first holomorphic De Rham cohomology group of $W$) and $g\in \Fcal_0(W)$ so that $\int_{\Gamma_2}\tau\neq 0,$ the function $f:=(\tau+dg)/d\xi$ lies in $\Fcal_0^*(S)\cap\Fcal_0(M_S)$ and $(f|_{M_S})\geq (\phi_2|_{M_S})^2.$
The existence of such 1-form and function follows from well known arguments on Riemann surfaces theory (take into account that (A) implies $d\xi|_{M_S}\not\equiv 0$). By Lemma \ref{lem:funaprox}, $f$ can be uniformly approximated on $S$ by functions in $\Sigma,$ so equation \eqref{eq:sobre2} applies and shows that $0=\int_{\Gamma_2} f d\xi=\int_{\Gamma_2} (\tau+dg)=\int_{\Gamma_2}\tau\neq 0,$ a contradiction. Therefore $\Gamma_2=0.$

Replacing $(\xi,\phi_1,\phi_2,\Gamma_1,\Gamma_2)$ by  $(1/\xi,\phi_2,\phi_1,\Gamma_2,\Gamma_1)$  and using a symmetric argument, we can prove that $\Gamma_1=0.$ This contradicts \eqref{eq:curvas} and  concludes the proof.
\end{proof}

Let $\{e_1,\ldots,e_{2\nu}\}$ be a basis of $\c^{2\nu},$  fix
$f_i \in \Acal^{-1}(e_i)\cap \Fcal_0(W)$ for all $i,$ and set $\Qcal:\c^{2 \nu} \to \c^{2 \nu}$ as the analytical map given by
\[
\Qcal((z_i)_{i=1,\ldots,2 \nu})=\Pcal(\sum_{i=1,\ldots,2\nu} z_i f_i).
\]

By Claim \ref{cla:sobre} the differential $d\Qcal_0$ of $\Qcal$ at $0$ is 
an isomorphism, then there exists a closed Euclidean ball $U\subset
\c^{2\nu}$ centered at the origin such that $\Qcal:U \to
\Qcal(U)$ is an analytical diffeomorphism. Furthermore,
notice that $0=\Qcal(0) \in \Qcal(U)$ is an interior
point of  $\Qcal(U).$

Consider a sequence $\{\theta_n\}_{n\in\n}\subset\Omega_0(W)$ uniformly approximating $\eta$ on $S$ and with $(\theta_n)=(\eta|_{M_S})$ for all $n$ (recall that $\eta|_{M_S}\not\equiv 0$ and see Lemma \ref{lem:formaprox}).

Label $\Pcal_n:\Fcal_0^*(S)\to \c^{2\nu}$ as the Fr\'{e}chet differentiable map given by
\[
\Pcal_n(f)=\left( \int_{c} \left( e^f\theta_n+e^{-f}\frac{\Theta}{\theta_n}-2\phi_1 \;,\; e^f\theta_n-e^{-f}\frac{\Theta}{\theta_n}+2\imath\phi_2 \right) \right)_{c \in \Bcal_S},\quad \forall n\in\n.
\]
Call $\Qcal_n:\c^{2\nu} \to \c^{2\nu}$ as the analytical map $\Qcal_n((z_i)_{i=1,\ldots,2\nu})=\Pcal_n(\sum_{i=1,\ldots,2 \nu} z_i f_i)$  for all $n \in \n.$ Since $\{\Qcal_n\}_{n \in \n} \to \Qcal$  uniformly on compacts subsets of $\c^{2\nu},$ without loss of generality we can suppose that $\Qcal_n:U \to \Qcal_n(U)$ is an analytical diffeomorphism  and  $0 \in \Qcal_n(U)$ for all $n.$ Label $\alpha_n=(\alpha_{1,n},\ldots,\alpha_{2\nu,n})$ as the unique point in $U$ such that $\Qcal_n (\alpha_n)=0$ and note that $\{\alpha_n\}_{n \in \n} \to 0.$ Set
\[
\eta_n:=e^{\sum_{i=1}^{2\nu} \alpha_{i,n} f_i} \theta_n,\quad 
\phi_{1,n}:=\frac12\left(\eta_n+\frac{\Theta}{\eta_n}\right)\quad \text{and} \quad \phi_{2,n}:=\frac{\imath}2\left(\eta_n-\frac{\Theta}{\eta_n}\right), \quad\forall n\in\n
\]
and let us check that the sequence $\{\Phi_n=(\phi_{1,n},\phi_{2,n})\}_{n \in\n }$ satisfies the conclusion of the lemma. Indeed, since $(\eta_n)=(\theta_n)=(\eta|_{M_S})$ one has $\Theta/\eta_n\in \Omega_0(W)$ and so $\Phi_n\in\Omega_0(W)^2.$ The convergence of $\{\Phi_n\}_{n\in\n}$ to $\Phi$ on $S$ follows from the ones of $\{\theta_n\}_{n\in\n}$ to $\eta$ and of $\{\alpha_n\}_{n\in\n}$ to $0.$ A straightforward computation gives (a). The fact that $\Qcal_n(\alpha_n)=0,$ $n\in\n,$ implies (b). Finally, $(\eta_n)=(\eta|_{M_S})$ for all $n$ implies (c).


The proof of the lemma in case (B)  goes as follows.

Notice that $\Theta= 0$ is nothing but $\phi_2= \beta\phi_1,$ where $\beta\in\{\imath,-\imath\}.$

As above, we can assume without loss of generality that $\phi_1$ never vanishes on $C_S$ (we omit the details). Reasoning as in case (A), we can prove that $\hat{\Acal}|_{\Fcal_0(W)}:\Fcal_0(W) \to \c^\nu$ is surjective, where $\hat{\Acal}$ is the Fr\'{e}chet derivative of $\hat{\Pcal}:\Fcal_0^*(S) \to \c^\nu,$  $\hat{\Pcal}(f)=\left(\int_c (e^f-1) \phi_1\right)_{c\in \hat{\Bcal}_S},$ at $0.$   Take $\hat{f}_i \in \hat{\Acal}^{-1}(\hat{e}_i)\cap \Fcal_0(W)$ for all $i$, where $\hat{\Bcal}_S=\{\hat{e}_1,\ldots, \hat{e}_\nu\}$ is a basis of $\c^\nu,$ and define $\hat{\Qcal}:\c^\nu \to \c^\nu$ by $\hat{\Qcal}((z_i)_{i=1,\ldots,\nu})=\hat{\Pcal}(\sum_{i=1,\ldots,\nu} z_i \hat{f}_i).$ Now, consider a sequence $\{\hat{\theta}_n\}_{n\in\n}\subset\Omega_0(W)$ that uniformly approximates $\phi_1$ on $S$ and $(\hat{\theta}_n)=(\phi_1|_{M_S})$ for all $n$ (as above, recall that $\phi_1|_{M_S}\not\equiv 0$ and see Lemma \ref{lem:formaprox}). 
Set $\hat{\Pcal}_n:\Fcal_0^*(S) \to \c^\nu$ by $\hat{\Pcal}_n(f)=\left(\int_c (e^f\hat{\theta}_n- \phi_1)\right)_{c\in \hat{\Bcal}_S},$ and call $\hat{\Qcal}_n:\c^\nu \to \c^\nu$ as the analytical map $\hat{\Qcal}_n((z_i)_{i=1,\ldots,\nu})=\hat{\Pcal}_n(\sum_{i=1,\ldots,\nu} z_i \hat{f}_i)$  for all $n \in \n.$  To finish, reason as in case (A) but setting $\phi_{1,n}:=e^{\sum_{i=1}^\nu \hat{\alpha}_{i,n} \hat{f}_i} \hat{\theta}_n$ and $\phi_{2,n}:=\beta \phi_{1,n},$ where $\hat{\alpha}_n=(\hat{\alpha}_{1,n},\ldots,\hat{\alpha}_{\nu,n})$ is chosen so that $\hat{\Qcal}_n(\hat{\alpha}_n)=0$ and $\{\hat{\alpha}_n\}_{n\in\n}\to 0.$
\end{proof}


\section{Main Results}\label{sec:main}

The main results of this paper follow as consequence of Lemma \ref{lem:main} below. Although the proof of this lemma is inspired by the technique developed in \cite[Lemma 3.1]{AFL}, it represents a wide generalization of that result.

We need the following notations and definitions. 

Fix a nowhere zero $\tau_0\in \Omega_0(\Ncal)$ (the existence of such a $\tau_0$ is well known, anyway see \cite{AFL} for a proof). Then for any compact subset $K\subset \Ncal$ and any $\theta\in\Omega_0^*(K)$ we set $\|\theta\|:=\max_K\{|\theta/\tau_0|\}.$ This norm induces the topology of the uniform convergence on $\Omega_0^*(K).$

Let $K\subset\Ncal$ be a connected compact region and $\sigma^2$ a Riemannian metric on $K$ possibly with singularities. Given $P,Q\in K$ we denote by $\dist_{(K,\sigma)}(P,Q)=\min\{\text{length}_\sigma(\alpha)\;|\; \alpha$ curve in $K$ joining $P$ and $Q\}.$ If $K_1$ and $K_2$ are two compact sets in $K$ we set $\dist_{(K,\sigma)}(K_1,K_2)=\min\{\dist_{(K,\sigma)}(P,Q)\;|\;P\in K_1,\, Q\in K_2\}.$

\begin{lemma}\label{lem:main}
Let $M,$ $K$ be two compact regions in $\Ncal$ with $M\subset K^\circ.$ Assume that $M$ is Runge, $K$ is connected and consider $P_0\in M^\circ.$  Let $\Ical$ be a conformal Riemannian metric on $K$ possibly with isolated singularities. Let $\ftt=(\ftt_1,\ftt_2):\Hcal_1(K,\z)\to\c^2$ be a group homomorphism, $\Theta\in \mho_0(K)$ and $\Phi=(\phi_1,\phi_2)\in \Omega_0(M)^2$ satisfying
\[
\phi_1^2+\phi_2^2=\Theta|_{M},\quad \text{$\ftt(\gamma)=\int_\gamma \Phi,$ $\forall \gamma\in\Hcal_1(M,\z),$}
\]
and either of the following conditions:
\begin{enumerate}[{\rm (A)}]
\item $\phi_1$ and $\phi_2$ are linearly independent in $\Omega_0(M).$
\item $\Theta= 0,$ $\phi_1\not\equiv 0$ and there is $\beta \in\{\imath,-\imath\}$ such that $\ftt_2=\beta\ftt_1$ and $\phi_2=\beta\phi_1.$
\end{enumerate}

Then, for any $\epsilon>0$ there exists $\Psi=(\psi_1,\psi_2)\in\Omega_0(K)^2$ so that
\begin{enumerate}[{\rm ({L}1)}]
\item $\|\Psi-\Phi\|<\epsilon$ on $M,$
\item $\psi_1^2+\psi_2^2=\Theta,$
\item $\ftt (\gamma)=\int_\gamma \Psi,$ $\forall \gamma \in\Hcal_1(K,\z),$
\item $\dist_{(K,\sigma_{(\Psi,\Ical)})}(P_0,\partial K)>1/\epsilon,$ where $\sigma_{(\Psi,\Ical)}^2:=|\psi_1|^2+|\psi_2|^2+\Ical,$ and
\item the zeros of $\Psi$ on $K$ are those of $\Phi$ on $M$ (in particular, $\Psi$ never vanishes on $K-M$).
\end{enumerate}
\end{lemma}
\begin{proof}
The proof goes by induction on minus the Euler characteristic of $W-M^\circ.$ Since $M$ is Runge then no component of $K -M^\circ$ is a closed disc, and so $-\chi(K -M^\circ)\geq 0.$ The basis of the induction is proved in the following

\begin{claim}\label{cla:base}
Lemma \ref{lem:main} holds if $\chi(K-M^\circ)=0.$
\end{claim}
\begin{proof}
In this case $K^\circ-M=\cup_{j=1}^k A_j$, where $A_j$ are pairwise disjoint open annuli, $k\in\n.$
On each $A_j$ we construct a Jorge-Xavier's type labyrinth of compact sets as follows (see \cite{JX}). Let $z_j:A_j\to\c$ be a conformal parametrization,
and let $C_j\subset A_j$ be a compact region such that $C_j$ contains no singularities of $\Ical,$ $z_j(C_j)$ is a compact annulus of radii $r_j$ and $R_j,$ where $r_j<R_j,$ and $z_j(C_j)$ contains the homology of $z_j(A_j).$ This choice is possible since the singularities of $\Ical$ are isolated. Since $\Ical|_{C_j}$ has no singularities, we can find a positive constant $\mu$ with 
\begin{equation}\label{eq:grande0}
\Ical>\mu^2|dz_j|^2\quad \text{ on $C_j,$ $j=1,\ldots,k.$}
\end{equation}
Consider a large $m\in\n$ (to be specified later) such that
$2/m<\min\{R_j-r_j\;|\;j=1,\ldots,k\}.$ For any $j\in\{1,\ldots,k\}$ label $s_{j,0}:=R_j$ and for any $n\in\{1,\ldots,2m^2\}$ set $s_{j,n}:=R_j-n/m^3$ and consider the compact set in $C_j$ (see Figure \ref{fig:laberinto}):
\[
\Kcal_{j,n}=\left\{ P\in A_j\; \left|\; s_{j,n}+\frac1{4m^3}\leq |z_j(P)|\leq
s_{j,n-1}-\frac1{4m^3},\right.\right.
\left.\frac1{m^2}\leq {\rm arg}((-1)^{n}z_j(P))
\leq 2\pi-\frac1{m^2}\right\}.
\]
\begin{figure}[ht]
    \begin{center}
    \scalebox{0.4}{\includegraphics{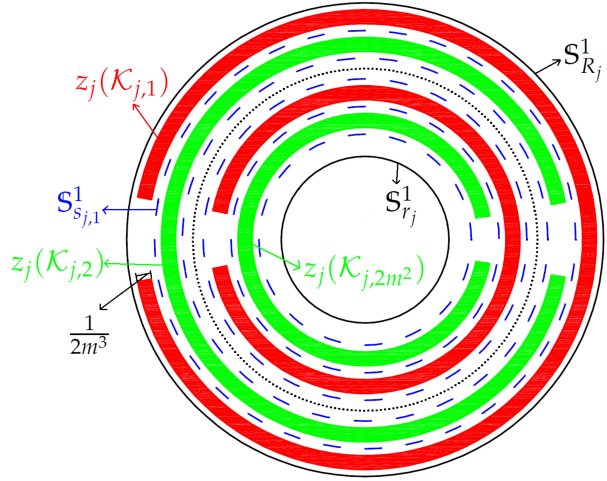}}
        \end{center}
\caption{The labyrinth of compact sets on the annulus $z_j(C_j).$}\label{fig:laberinto}
\end{figure}

Then, define
\[
\Kcal_j=\bigcup_{n=1}^{2m^2}\Kcal_{j,n}\qquad \mbox{and}\qquad  \Kcal=\bigcup_{j=1}^k \Kcal_j.
\]

Consider the pair $\Xi=(\varphi_1,\varphi_2)\in \Omega_0(M\cup \Kcal)^2$ given by
\[
\Xi|_M=\Phi, \qquad \Xi|_{\Kcal_j}=
\begin{cases}
\big(\frac{1}{2}(\lambda dz_j + \frac{\Theta}{\lambda dz_j} )\,,\, \frac{\imath}{2}(\lambda dz_j - \frac{\Theta}{\lambda dz_j})\big)& \text{if (A) holds}
\\
\big(\lambda dz_j\,,\, \beta\lambda dz_j\big)& \text{if (B) holds},
\end{cases}\quad j=1,\ldots,k,
\]
where $\lambda>\sqrt{2}\,\mu m^4$ is a constant. Notice that $\varphi_1^2+\varphi_2^2= \Theta|_{M\cup\Kcal}.$

Let $W\subset\Ncal$ be a domain with finite topology containing $K.$ Applying Lemma \ref{lem:ap} to the data
\[
\hat{W}=W,\quad \hat{S}=M\cup \Kcal,\quad \hat{\Theta}=\Theta,\quad\text{and}\quad \hat{\Phi}=\Xi,
\] 
we obtain a pair $\Psi=(\psi_1,\psi_2)\in\Omega_0(K)^2$ satisfying (L1), (L2), (L3), (L5) and
\begin{equation}\label{eq:grande}
|\psi_1|^2+|\psi_2|^2 > 
\mu^2 m^8 |dz_j|^2\quad \text{ on $\Kcal_j,$ $j=1,\ldots,k.$}
\end{equation}

Then, taking into account \eqref{eq:grande0}, \eqref{eq:grande} and the definition of $\Kcal_j,$ it is straightforward to check the existence of a positive constant $\rho_j$ depending neither on $\mu$ nor $m$ such that
\[
{\rm length}_{\sigma_{(\Psi,\Ical)}}(\alpha)>\rho_j\cdot \mu\cdot m
\]
for any $\alpha$ curve in $C_j$ joining the two components of $\partial C_j$. Thus, we can choose $m$ large enough so that
$\rho_j\cdot \mu\cdot m>1/\epsilon$ for any $j=1,\ldots,k$. This choice gives (L4) and we are done.
\end{proof}

The inductive step and so Lemma \ref{lem:main} are proved in the following

\begin{claim}\label{cla:paso}
Consider $n > 0$ and assume that Lemma \ref{lem:main} holds if $-\chi(K -M^\circ)< n.$ Then it also holds if
$-\chi(K-M^\circ)=n.$
\end{claim}
\begin{proof} Since $M$ is Runge, $\jmath_*:\Hcal_1(M,\z) \to \Hcal_1(K,\z)$ is a monomorphism, where $\jmath:M \to K$ is the inclusion map. Up to this natural identification we will  consider $\Hcal_1(M,\z) \subset \Hcal_1(K,\z).$ 
Since $-\chi(K-M^\circ)=n>0$, there exists $\hat\gamma\in\Hcal_1(K,\z)-\Hcal_1(M,\z)$ intersecting $K-M^\circ$
in a compact Jordan arc $\gamma$ with endpoints $P_1,P_2\in\partial M$ and otherwise disjoint from $\partial M\cup \partial K,$ and such that $S:=M\,\cup\,\gamma$ is admissible. Notice that in this case $\gamma=C_S$ and $M=M_S.$

Assume (A) holds, and in addition choose $\hat{\gamma}$ so that $\Theta$ never vanishes on $\gamma.$ 
Consider a pair $\hat{\Phi}=(\hat{\phi}_1,\hat{\phi}_2)\in \Omega_0^*(S)^2 \cap \Omega_0(M_S)^2$ satisfying
$\hat{\Phi}|_M=\Phi,$ $\hat{\phi}_1^2+\hat{\phi}_2^2= \Theta|_S$ and $\int_{\hat{\gamma}}\hat{\Phi}=\ftt(\hat{\gamma})$ (we leave the details to the reader). By Lemma \ref{lem:ap}, case (A), applied to $\hat{\Phi},$ $S,$ $\Theta$ and $K^\circ,$ we can find a compact tubular neighborhood  $U$ of $S$ in $K^\circ$ and $\Xi=(\varphi_1,\varphi_2)\in\Omega_0(U)^2$ such that $\varphi_1$ and $\varphi_2$ are linearly independent in $\Omega_0(U)^2,$
$\|\Xi-\Phi\|<\epsilon/2$ on $M,$ $\varphi_1^2+\varphi_2^2= \Theta|_{U},$ the zeros of $\Xi$ on $U$ are those of $\Phi$ on $M,$ and $\Xi-\hat{\Phi}$ is exact on $S.$ 
Since $-\chi(K-U^\circ)<n$, the induction hypothesis applied to $\Xi$ and $\epsilon/2$ gives the existence of a pair $\Psi\in\Omega_0(K)^2$ satisfying the conclusion of the lemma.

Assume now that (B) holds, and take a function $\hat{\phi}_1\in\Omega_0^*(S)\cap \Omega_0(M_S)$ such that $\hat{\phi}_1|_M=\phi_1$ and $\int_{\hat{\gamma}}\hat{\phi}_1=\ftt_1(\hat{\gamma}).$ Apply Lemma \ref{lem:ap}, case (B), to the data $K^\circ,$ $S$ and $(\hat{\phi}_1,\beta\hat{\phi}_1),$ and obtain a compact tubular neighborhood $U$ of $S$ in $K^\circ$ and a 1-form $\varphi_1\in\Omega_0(U)$ such that $\|\varphi_1-\phi_1\|<\epsilon/4$ on $M,$ the zeros of $\varphi_1$ on $U$ are those of $\phi_1$ on $M,$ and $\varphi_1-\hat{\phi_1}$ is exact on $S.$ As above, the induction hypothesis applied to $(\varphi_1,\beta\varphi_1)$ and $\epsilon/2$ gives a pair $\Psi\in\Omega_0(K)^2$ proving the claim.
\end{proof}
This finishes the proof of the lemma.
\end{proof}


Now we can state and prove the main theorem of this paper.

\begin{theorem}\label{th:main}
Let $M\subset \Ncal$ be a Runge compact region. Let $\Ical$ be a conformal Riemannian metric on $\Ncal$ possibly with isolated singularities. Consider $\ftt=(\ftt_1,\ftt_2):\Hcal_1(\Ncal,\z)\to\c^2$ be a group homomorphism, $\Theta\in \mho_0(\Ncal)$ and $\Phi=(\phi_1,\phi_2)\in \Omega_0(M)^2$ satisfying
\[
\phi_1^2+\phi_2^2=\Theta|_{M},\quad \text{$\ftt(\gamma)=\int_\gamma \Phi,$ $\forall \gamma\in\Hcal_1(M,\z),$}
\]
and either of the following conditions:
\begin{enumerate}[{\rm (A)}]
\item $\phi_1$ and $\phi_2$ are linearly independent in $\Omega_0(M).$
\item $\Theta= 0,$ $\phi_1\not\equiv 0$ and there is $\beta \in\{\imath,-\imath\}$ such that $\ftt_2=\beta\ftt_1$ and $\phi_2=\beta\phi_1.$
\end{enumerate}

Then, for any $\epsilon>0$ there exists $\Psi=(\psi_1,\psi_2)\in\Omega_0(\Ncal)^2$ so that
\begin{enumerate}[{\rm ({T}1)}]
\item $\|\Psi-\Phi\|<\epsilon$ on $M,$
\item $\psi_1^2+\psi_2^2=\Theta,$
\item $\ftt (\gamma)=\int_\gamma \Psi,$ $\forall \gamma \in\Hcal_1(\Ncal,\z),$
\item $|\psi_1|^2+|\psi_2|^2+\Ical$ is a complete conformal Riemannian metric on $\Ncal$ with singularities at the zeros of $|\phi_1|^2+|\phi_2|^2+\Ical$ on $M,$ and
\item the zeros of $\Psi$ on $\Ncal$ are those of $\Phi$ on $M$ (in particular, $\Psi$ never vanishes on $\Ncal-M$).
\end{enumerate}
\end{theorem}
\begin{proof}
Label $M_1=M$ and let $\{M_n\,|\; n \geq 2\}$ be an exhaustion of $\Ncal$ by Runge connected compact regions with $M_n \subset M_{n+1}^\circ$ for all $n\in\n.$ Fix a base point $P_0\in M^\circ$ and a positive $\varepsilon<\min\{\epsilon,1\}$ which will be specified later.  

Label $\Phi_1=\Phi,$ and by Lemma \ref{lem:main} and an inductive
process, construct a sequence of pairs $\{\Phi_n=(\phi_{j,n})_{j=1,2}\}_{n \in \n}$ satisfying that
\begin{enumerate}[{\rm (a)}]
  \item $\Phi_n \in \Omega_0(M_n)^2,$ $\forall n \in \n,$
  \item $\|\Phi_n-\Phi_{n-1}\|< \varepsilon/2^n$ on $M_{n-1},$ $\forall n \geq 2,$
  \item $\phi_{1,n}^2+\phi_{2,n}^2=\Theta|_{M_n},$ $\forall n\in\n,$
  \item $\ftt(\gamma)=\int_\gamma\Phi_n,$ $\forall\gamma\in \Hcal_1(M_n,\z),$ $\forall n \in \n,$
  \item $\dist_{(M_n,\sigma_{(\Phi_n,\Ical)})}(P_0,\partial M_n)>2^n,$ where $\sigma_{(\Phi_n,\Ical)}^2=|\phi_{1,n}|^2+|\phi_{2,n}|^2+\Ical,$ $\forall n \geq 2,$ and
  \item the zeros of $\Phi_n$ on $M_n$ are those of $\Phi$ on $M,$ $\forall n\in\n.$ 
\end{enumerate}

Since $\cup_{n\in\n}M_n=\Ncal,$ items (a) and (b) and Harnack's theorem, then the sequence $\{\Phi_n\}_{n\in
\n}$ uniformly converges on compact subsets of $\Ncal$ to a
pair $\Psi=(\psi_j)_{j=1,2}\in\Omega_0(\Ncal)$ satisfying (T1). Items (c) and (d) directly give (T2) and (T3), respectively. Since $\{\Phi_n\}_{n\in\n}$ uniformly converges to $\Psi$ and (f), Hurwitz's theorem gives that
either the zeros of $\Psi$ on $\Ncal$ are those of $\Phi$ on $M$ or $\psi_1=0$ or $\psi_2=0.$
However, (b) gives $\|\Psi-\Phi\|\leq\varepsilon$ on $M$ and so $\psi_j|_M\not\equiv 0,$ $j=1,2,$ provided that $\varepsilon$ is small enough. This proves (T5). Finally (T5) and (e) imply (T4) and we are done.
\end{proof}


\begin{corollary}\label{co:harmonic}
Let $\Hgot,$ $X=(X_i)_{i=3,\ldots,\nsc}:\Ncal\to\r^{\nsc-2}$ and $\ptt=(\ptt_j)_{j=1,\ldots,\nsc}:\Hcal_1(\Ncal,\z)\to\r^\nsc$ be a 2-form in $\mho_0(\Ncal),$ a non-constant harmonic map and   a group homomorphism, respectively, satisfying that
\begin{itemize}
\item $\ptt_i(\gamma)={\rm Im}\int_\gamma \partial_z X_i,$ $\forall \gamma\in\Hcal_1(\Ncal,\z),$ $\forall i=3,\ldots,\nsc,$ and
\item $\ptt_1= \ptt_2=0$ when $\Hgot= \sum_{i=3}^\nsc (\partial_z X_i)^2.$
\end{itemize}

Then there exists a weakly complete harmonic map $Y=(Y_j)_{j=1,\ldots,\nsc}:\Ncal\to\r^\nsc$ with
\begin{enumerate}[{\rm (I)}]
\item $(Y_i)_{i=3,\ldots,\nsc}=X,$
\item $\ptt_Y=\ptt,$ and
\item $Q_Y=\Hgot.$
\end{enumerate}

Furthermore, if $X$ is full then $Y$ can be chosen to be full, and if $X$ is an immersion then $Y$ is.
\end{corollary}

\begin{proof}
Label $\Theta:=\Hgot-\sum_{i=3}^\nsc (\partial_z X_i)^2,$ and assume for a moment that ${\Theta}\not\equiv 0.$ Consider a compact disc $K\subset \Ncal$ and $\eta\in\Omega_0(K)$ such that both $\eta$ and $\phi_1$ never vanish on $K,$ and $\phi_1$ and $\phi_2$ are linearly independent in $\Omega_0(K),$ where $\phi_1:=\frac12(\eta+{\Theta}/\eta)$ and $\phi_2:=\frac{\imath}2(\eta-{\Theta}/\eta).$ Consider a pair $\Psi=(\psi_1,\psi_2)$ obtained from Theorem \ref{th:main}, case (A), applied to the data
\[
\Ncal,\quad M=K,\quad \Ical=\sum_{i=3}^\nsc |\partial_z X_i|^2,\quad {\Theta},\quad \Phi=(\phi_1,\phi_2),\quad \ftt=\imath(\ptt_1,\ptt_2)
\]
and $\epsilon>0$ to be specified later. Fix a point $P_0\in\Ncal$ and define $Y_k(P)={\rm Re}\int_{P_0}^P \psi_k,$ $\forall P\in \Ncal,$ $k=1,2,$ and $Y_k=X_k,$ $\forall\, k=3,\ldots,\nsc.$

Statements (I), (II) and (III) trivially follow from the definition of ${\Theta}$ and $\ftt,$ and properties (T2) and (T3). Moreover, (T4) and the fact that $\phi_1$ never vanishes on $K$ give that $\sum_{j=1}^\nsc|\partial_z Y_j|^2$ is a complete conformal metric on $\Ncal,$ and so $Y$ is weakly complete. Finally, if $X$ is full then we can choose $\eta$ so that the map
\[
K\to\r^\nsc,\quad P\mapsto \left( \int_{P_0}^P\phi_1\;,\; \int_{P_0}^P\phi_2\;,\;X(P)\right)
\]
is full as well. Then (T1) gives the fullness of $Y$ provided that $\epsilon$ is chosen small enough.

Assume now that ${\Theta}= 0.$ Take an exact $\phi_1\in\Omega_0(M),$ $\phi_1\not\equiv 0,$ and consider a pair $\Psi$ obtained by applying Theorem \ref{th:main}, case (B), to the data
\[
\Ncal,\quad M=K,\quad \Ical=\sum_{i=3}^\nsc |\partial_z X_i|^2,\quad {\Theta}= 0,\quad \Phi=(\phi_1,\imath\phi_1),\quad \ftt=0
\]
and $\epsilon>0.$ To finish argue as above.
\end{proof}

\begin{corollary}\label{co:minimal}
Let $X=(X_i)_{i=3,\ldots,n}:\Ncal\to\r^{\nsc-2}$ and $\ptt=(\ptt_j)_{j=1,\ldots,\nsc}:\Hcal_1(\Ncal,\z)\to\r^\nsc$ be a non-constant harmonic map and a group homomorphism, respectively, satisfying that
\begin{itemize}
\item $\ptt_i(\gamma)={\rm Im}\int_\gamma \partial_z X_i,$ $\forall \gamma\in\Hcal_1(\Ncal,\z),$ $\forall i=3,\ldots,\nsc,$ and
\item $\ptt_1= \ptt_2=0$ when  $\sum_{i=3}^\nsc (\partial_z X_i)^2= 0.$
\end{itemize}

Then there exists a complete conformal minimal immersion $Y=(Y_j)_{j=1,\ldots,\nsc}:\Ncal\to\r^\nsc$ with
$(Y_i)_{i=3,\ldots,\nsc}=X$ and $\ptt_Y=\ptt.$ Furthermore, $Y$ can be chosen full provided that $X$ is.
\end{corollary}
\begin{proof}
Apply Corollary \ref{co:harmonic} for $\Hgot= 0$ and see Remark \ref{rem:Q=0}.
\end{proof}


\begin{corollary}\label{co:CM}
Let $\Ncal$ be a bounded planar domain. Then there exists a complete non-proper holomorphic embedding of $\Ncal$ in $\c^2.$
\end{corollary}
\begin{proof}
Consider $X=(X_3,X_4):\Ncal\to\r^2\equiv \c$ given by $X(z)=z.$ Let $Y=(Y_j)_{j=1,\ldots,4}:\Ncal\to\r^4$ be an immersion obtained from Corollary \ref{co:minimal} applied to the data $\Ncal,$ $X$ and $\ptt= 0.$ Since $X$ is injective, $Y$ is an embedding. Finally, observe that $Y$ is non-proper. Indeed, otherwise the holomorphic function $Y_1+\imath Y_2$ would be proper on $\Ncal,$ contradicting that  $\Ncal$ is hyperbolic.
\end{proof}


\begin{corollary}\label{co:omite}
Let $\ptt:\Hcal_1(\Ncal,\z)\to\r^\nsc$ be a group homomorphism.

Then there exists a conformal complete minimal immersion $Y:\Ncal\to\r^\nsc$ satisfying
\begin{itemize}
\item $\ptt_Y=\ptt,$
\item $Y$ is non-decomposable and full,
\item $G_Y$ is non-degenerate, and
\item $G_Y$ fails to intersect $\nsc$ hyperplanes of $\cp^{\nsc-1}$ in general position.
\end{itemize}
\end{corollary}
\begin{proof}
We need the following
\begin{claim}[\text{\cite[Theorem 4.2]{AFL}}]\label{cla:noceros}
For any group homomorphism $\hat{\ptt}:\Hcal_1(\Ncal,\z)\to\r$ there exists a never vanishing $\phi\in\Omega_0(\Ncal)$ with $\int_\gamma\phi=\imath \hat{\ptt}(\gamma),$ $\forall \gamma \in \Hcal_1(\Ncal,\z).$
\end{claim}

Assume first that $\nsc$ is even.

Consider a nowhere zero $\phi\in\Omega_0(\Ncal)$ (see Claim
\ref{cla:noceros}) and a compact disc $M\subset\Ncal.$ Fix $P_0\in
M^\circ$ and take $\lambda_j\in \c-\{0\}$ and $\Phi_j=(\phi_{j,1},\phi_{j,2})\in
\Omega_0(M)^2,$ $j=1,\ldots, \nsc/2,$ so that
\begin{itemize}
\item $\sum_{j=1}^{\nsc/2} \lambda_j^2=0,$
\item  $\phi_{j,1}$ and $\phi_{j,2}$ are linearly independent in $\Omega_0(M)$ and  $\phi_{j,1}^2+\phi_{j,2}^2=\lambda_j^2\phi^2|_M,$ $\forall j=1,\ldots, \nsc/2,$
\item the minimal immersion $X:M\to\r^\nsc,$  $X(P)={\rm
Re}(\int_{P_0}^P (\Phi_j)_{j=1,\ldots,\nsc/2})$ is non-decomposable and full, and
\item $G_X$ is non-degenerate.
\end{itemize}
Write $\ptt=(\ptt_k)_{k=1,\ldots,\nsc},$ and for any $j=1,\ldots,\nsc/2$
consider $\Psi_j=(\psi_{j,1},\psi_{j,2})\in\Omega_0(\Ncal)^2$ given
by Theorem \ref{th:main}, case (A), applied to the data
\[
\Ncal,\quad M,\quad \Ical=|\phi|^2,\quad \ftt=\imath(\ptt_{2j-1},\ptt_{2j}),\quad \Theta=\lambda_j^2\phi^2,\quad
\Phi=\Phi_j,
\]
and $\epsilon>0$ which will be specified later. Define
\[
Y:\Ncal\to\r^\nsc,\quad Y(P)={\rm Re}\left(\int_{P_0}^P (\Psi_j)_{j=1,\ldots,\nsc/2}\right).
\]

Statement (T3) in Theorem \ref{th:main} gives that $Y$ is well
defined. From (T2) follows that $\sum_{j=1}^{\nsc/2}
(\psi_{j,1}^2+\psi_{j,2}^2)= 0,$ and so $Y$ is conformal.
Moreover, $\sum_{j=1}^{\nsc/2} (|\psi_{j,1}|^2+|\psi_{j,2}|^2)\geq
|\psi_{1,1}|^2+|\psi_{1,2}|^2\geq
\frac1{|\lambda_1|^2+1}(|\psi_{1,1}|^2+|\psi_{1,2}|^2+|\phi|^2)$ that is a complete Riemannian metric on $\Ncal$ (take
into account (T4)).
Therefore, $Y$ is a complete conformal minimal immersion. Item (T3)
implies that $\ptt_Y=\ptt.$ Since $X$ is non-decomposable and full
and $G_X$ is non-degenerate, then  $Y$ and $G_Y$ are, provided that $\epsilon$
is chosen small enough (see (T1)). Finally, observe that $\psi_{j,1}^2+\psi_{j,2}^2$ never
vanishes on $\Ncal$ for all $j=1,\ldots,\nsc/2,$ hence $G_Y$ fails
to intersect the hyperplanes
\[
\Pi_{j,\delta}:=\left\{ [(w_k)_{k=1,\ldots,\nsc}]\in\cp^{\nsc-1}\;|\;
w_{2j-1}+(-1)^\delta\imath w_{2j}=0 \right\},\quad \forall
(j,\delta)\in\{1,\ldots,\nsc/2\}\times\{0,1\},
\]
which are located in general position.

Assume now that $\nsc$ is odd.

Write $\ptt=(\ptt_k)_{k=1,\ldots,\nsc}$ and consider a nowhere zero
$\phi\in\Omega_0(\Ncal)$ with $\int_\gamma\phi=\imath
\ptt_\nsc(\gamma),$ $\forall\gamma\in\Hcal_1(\Ncal,\z)$ (see Claim
\ref{cla:noceros}). Fix a compact disc $M\subset\Ncal$ and a point
$P_0\in M^\circ.$ Take $\lambda_j\in\c-\{0\}$ and 
$\Phi_j=(\phi_{j,1},\phi_{j,2})\in\Omega_0(M)^2,$ $j=1,\ldots,(\nsc-1)/2$ so that:
\begin{itemize}
\item $\sum_{j=1}^{(\nsc-1)/2}\lambda_j^2=-1,$
\item $\phi_{j,1}$ and $\phi_{j,2}$ are linearly independent in $\Omega_0(M)$ and  $\phi_{j,1}^2+\phi_{j,2}^2=\lambda_j^2\phi^2|_M,$ $\forall j=1,\ldots, (\nsc-1)/2,$
\item the minimal immersion $X:M\to\r^\nsc,$  $X(P)={\rm
Re}(\int_{P_0}^P ((\Phi_j)_{j=1,\ldots,(\nsc-1)/2},\phi)$ is non-decomposable and full, and 
\item $G_X$ is non-degenerate.
\end{itemize}

For any $j=1,\ldots,(\nsc-1)/2$ consider
$\Psi_j=(\psi_{j,1},\psi_{j,2})\in\Omega_0(\Ncal)^2$ given by
Theorem \ref{th:main}, case (A), applied to the data
\[
\Ncal,\quad M,\quad \Ical=|\phi|^2,\quad \ftt=\imath(\ptt_{2j-1},\ptt_{2j}),\quad \Theta=\lambda_j^2\phi^2,\quad
\Phi=\Phi_j,
\]
and $\epsilon>0$ which will be specified later.

As above
\[
Y:\Ncal\to\r^\nsc,\quad Y(P)={\rm Re}\left(\int_{P_0}^P
((\Psi_j)_{j=1,\ldots,(\nsc-1)/2},\phi)\right)
\]
is the immersion we are looking for, provided that $\epsilon$ is small enough. In this
case $G_Y$ fails to intersect the following hyperplanes of
$\cp^{\nsc-1}$ located in general position:
\[
\Pi_{j,\delta}:=\left\{ [(w_k)_{k=1,\ldots,\nsc}]\in\cp^{\nsc-1}\;|\;
w_{2j-1}+(-1)^\delta\imath w_{2j}=0 \right\},
\]
$\forall (j,\delta)\in\{1,\ldots,(\nsc-1)/2\}\times\{0,1\},$ and
\[
\Pi:=\left\{ [(w_k)_{k=1,\ldots,\nsc}]\in\cp^{\nsc-1}\;|\; w_\nsc=0
\right\}.
\]

The proof is done.
\end{proof}


\end{document}